\documentclass[12pt,a4paper]{article}
\usepackage{a4wide}
\usepackage{amsmath,amsthm,amssymb}
\usepackage{array,enumerate}
\usepackage{indentfirst}	
\usepackage{authblk, color,hyperref}
\usepackage{graphicx}

\newtheorem{theorem}{Theorem}[section]

\newtheorem{lemma}[theorem]{Lemma}

\newtheorem{remark}[theorem]{Remark}
\usepackage{tikz}
\usetikzlibrary{shapes}
\newcommand{\g}{\Gamma}
\newcommand{\z}{\zeta}

\title{Non-bipartite distance-regular graphs with a small smallest eigenvalue}
\author[a]{Zhi Qiao}
\author[b]{Yifan Jing}
\author[c]{Jack Koolen\thanks{Corresponding author}\setcounter{footnote}{-1}\footnote{E-mail addresses: zhiqiao@sicnu.edu.cn (Z. Qiao), yifanjing17@gmail.com (Y. Jing), koolen@ustc.edu.cn (J. Koolen)}}
\affil[a]{School of Mathematical Sciences, Sichuan Normal University, 610068, Sichuan, PR China}
\affil[b]{Department of Mathematics, University of Illinois at Urbana Champaign, Urbana, IL, 61801, USA}
\affil[c]{School of Mathematical Sciences, University of Science and Technology of China, Wen-Tsun Wu Key Laboratory of the Chinese Academy of Sciences, 230026, Anhui, PR China}
\date{}

\begin{document}
\maketitle

\begin{abstract}
In 2017, Qiao and Koolen showed that 
for any fixed integer $D\geq 3$, there are only finitely many such graphs with $\theta_{\min}\leq -\alpha k$, where $0<\alpha<1$ is any fixed number. 
In this paper, we will study non-bipartite distance-regular graphs with relatively small $\theta_{\min}$ compared with $k$. 
In particular, we will show that if $\theta_{\min}$ is relatively close to $-k$, then the odd girth $g$ must be large. Also we will classify the non-bipartite distance-regular graphs with $\theta_{\min} \leq \frac{D-1}{D}$ for $D =4,5$.


\end{abstract}

\section{Introduction}
The {\em odd girth} of a non-bipartite graph is the length of its shortest odd cycle. Let $\g$ be a non-bipartite distance-regular graph with valency $k$, diameter $D$, odd girth $g$ and smallest eigenvalue $\theta_{\min}$. 
In \cite{KQ17}, Qiao and Koolen showed that 
for any fixed integer $D\geq 3$, there are only finitely many such graphs with $\theta_{\min}\leq -\alpha k$, where $0<\alpha<1$ is any fixed number. 
In this paper, we will study non-bipartite distance-regular graphs with relatively small $\theta_{\min}$ compared with $k$. 
In the next result, we will show that if $\theta_{\min}$ is relatively close to $-k$, then the odd girth $g$ must be large. 
\begin{theorem} 	\label{thm:main1}
	Let $\g$ be a non-bipartite distance-regular graph with valency $k$ and odd girth $g$, having smallest eigenvalue $\theta_{\min}$. 
	Then there exists a constant $\varepsilon (g)>0$ such that $\theta_{\min}\geq -(1-\varepsilon(g))k$.
\end{theorem}
\begin{remark}
	The positive constant $\varepsilon(g)$ goes to $0$ as the odd girth $g$ goes to $\infty$. For example, the $(2t+1)$-gon has valency $k=2$, odd girth $g=2t+1$ and smallest eigenvalue $\theta_{\min}=2\cos(\frac{2t\pi}{2t+1})$. Thus, $\varepsilon(g) \leq 1+\frac{\theta_{\min}}{k}=2\cos^2(\frac{t\pi}{2t+1})$. 
\end{remark}

In \cite{KQ17}, Qiao and Koolen classified non-bipartite distance-regular graphs with valency $k$, diameter $D\leq 3$ and smallest eigenvalue $\theta_{\min}\leq -k/2$. 
Using Theorem \ref{thm:main1}, we will classify non-bipartite distance-regular graphs with valency $k$, diameter $D$ and smallest eigenvalue $\theta_{\min}\leq -\frac{D-1}{D}k$, when $D=4$ or $5$. 

\begin{theorem}	\label{thm:main2}
	Let $\g$ be a non-bipartite distance-regular graph with valency $k$, diameter $D$ and smallest eigenvalue $\theta_{\min}\leq -\frac{D-1}{D}k$. 
	\begin{enumerate}[i)]
		\item If $D=4$, then $\g$ is one of the following graph
		\begin{enumerate}[a)]
			\item the Coxeter graph with intersection array $\{3,2,2,1;1,1,1,2\}$,
			\item the 9-gon with intersection array $\{2,1,1,1;1,1,1,1\}$,
			\item the Odd graph $O_5$ with intersection array $\{5,4,4,3;1,1,2,2\}$,
			\item the folded $9$-cube with intersection array $\{9,8,7,6;1,2,3,4\}$.
		\end{enumerate}
		\item If $D=5$, then $\g$ is one of the following graph
		\begin{enumerate}[a)]
			\item the 11-gon with intersection array $\{2,1,1,1,1;1,1,1,1,1\}$,
			\item the Odd graph $O_6$ with intersection array $\{6,5,5,4,4;1,1,2,2,3\}$,
			\item the folded $11$-cube with intersection array $\{11,10,9,8,7;1,2,3,4,5\}$.
		\end{enumerate}
	\end{enumerate}
\end{theorem}

This paper is organized as follows. 
In the next section, we give the definitions and some preliminary results. 
In section \ref{sec:main1}, we give a proof of Theorem \ref{thm:main1}. 
In the last section, we give a proof of Theorem \ref{thm:main2}. 

\section{Preliminaries}

For more background, see \cite{BCN} and \cite{DKT}. 

All the graphs considered in this paper are finite, undirected and simple. 
Let $\Gamma$ be a graph with vertex set $V=V(\Gamma)$ and edge set $E=E(\Gamma)$. 
Denote $x\sim y$ if the vertices $x,y\in V$ are adjacent. 
The {\em distance} $d(x,y)=d_\Gamma(x,y)$ between two vertices $x,y\in V(\g)$ is the length of a shortest path connecting $x$ and $y$. 
The maximum distance between two vertices in $\Gamma$ is the {\em diameter} $D=D(\Gamma)$. 
We use $\Gamma_i(x)$ for the set of vertices at distance $i$ from $x$ and write, for the sake of simplicity, $\Gamma(x):=\Gamma_1(x)$. 
The {\em degree} of $x$ is the number $|\Gamma(x)|$ of vertices adjacent to it. 
A graph is {\em regular with valency $k$} if the degree of each of its vertices is $k$.
The {\em girth} and {\em odd girth} of a graph is the length of its shortest cycle, and shortest odd cycle, respectively. 
A graph Γ is called {\em bipartite} if it has no odd cycle.

A connected graph $\Gamma$ with diameter $D$ is called {\em distance-regular} if there are integers $b_i$, $c_i$ ($i=0,1,\ldots,D$) such that for any two vertices $x,y\in V(\Gamma)$ with $d(x,y)=i$, there are exactly $c_i$ neighbors of $y$ in $\Gamma_{i-1}(x)$ and $b_i$ neighbors of $y$ in $\Gamma_{i+1}(x)$, where we define $b_D=c_0=0$. 
In particular, $\Gamma$ is a regular graph with valency $k:=b_0$. 
We define $a_i:=k-b_i-c_i$ $(i=0,1,\ldots,D)$ for notational convenience. Note that $a_i =|\Gamma(y)\cap\Gamma_i(x)|$ holds for any two vertices $x,y$ with $d(x,y)=i$ $(i=0,1,\ldots,D)$.

For a distance-regular graph $\Gamma$ and a vertex $x\in V(\Gamma)$, we denote $k_i:=|\Gamma_i(x)|$ and $p_{ij}^h:=|\{w\mid w\in \g_i(x)\cap\g_j(y)\}|$ for any $y\in\g_h(x)$. 
It is easy to see that $k_i = b_0b_1\cdots b_{i-1}/(c_1c_2\cdots c_i)$ and hence it does not depend on $x$. 
The numbers $a_i$ , $b_i$ and $c_i$ ($i=0,1,\ldots,D$) are called the {\em intersection numbers}, and the array $\{b_0,b_1,\ldots,b_{D-1};c_1,c_2,\ldots,c_D\}$ is called the {\em intersection array} of $\g$. 
The matrix $L$ is called the {\em intersection matrix} of $\g$, where 
$$
L=\begin{pmatrix}
	a_0 & b_0 & 0     &       &       &         \\
	c_1 & a_1  & b_1    &       & 0     &         \\
	    & c_2 & a_2    & \cdot &       &         \\
	    &     & \cdot & \cdot & \cdot &         \\
	    & 0   &       & \cdot & \cdot & b_{D-1} \\
	    &     &       &       & c_D   & a_D
\end{pmatrix}. 
$$

Let $\g$ be a distance-regular graph with $v$ vertices and diameter $D$. 
Let $A_i$ $(i=0,1,\ldots,D)$ be the $(0,1)$-matrix whose rows and columns are indexed by the vertices of $\g$ and the $(x, y)$-entry is $1$ whenever $d(x,y)=i$ and $0$ otherwise. 
We call $A_i$ the {\em distance-$i$ matrix} and $A:=A_1$ the {\em adjacency matrix} of $\g$. 
The {\em eigenvalues} $\theta_0>\theta_1>\cdots>\theta_D$ of the graph $\Gamma$ are just the eigenvalues of its adjacency matrix $A$. 
We denote $m_i$ the {\em multiplicity} of $\theta_i$. 
Note that the $D+1$ distinct eigenvalues of $\g$ are precisely the eigenvalues of $L$ (see \cite[Proposition 2.7]{DKT}). 

For each eigenvalue $\theta_i$ of $\Gamma$, let $U_i$ be a matrix with its columns forming an orthonormal basis for the eigenspace associated with $\theta_i$. 
And $E_i:=U_i U_i^T$ is called the {\em minimal idempotent} associated with $\theta_i$, satisfying $E_i E_j= \delta_{ij} E_j$ and $A E_i=\theta_i E_i$, where $\delta_{ij}$ is the Kronecker delta. 
Note that $vE_0$ is the all-ones matrix $J$. 

The set of distance matrices $\{A_0=I,A_1,A_2,\ldots,A_D\}$ forms a basis of a commutative $\mathbb R$-algebra $\mathcal A$, known as the {\em Bose-Mesner algebra}. 
The set of minimal idempotents $\{E_0=\frac{1}{v}J,E_1,E_2,\ldots,E_D\}$ is another basis of $\mathcal A$. 
There exist $(D+1)\times(D+1)$ matrices $P$ and $Q$ (see \cite[p.45]{BCN}), such that the following relations hold 
\begin{equation}	\label{eq:pq}
A_i=\sum_{j=0}^D P_{ji}E_j \quad \text{and} \quad E_i=\frac{1}{v}\sum_{j=0}^D Q_{ji}A_j\quad (i=0,1,\ldots,D).
\end{equation}
Note that  $Q_{0i}=m_i$ (see \cite[Lemma 2.2.1]{BCN}). 

Let $E_i=U_i U_i^T$ be the minimal idempotent associated with $\theta_i$, where the columns of $U_i$ form an orthonormal basis of the eigenspace associated with $\theta_i$. 
We denote the $x$-th row of $\sqrt{v/m_i}U_i$ by $\hat{x}$. 
Note that $E_i\circ A_j=\frac{1}{v}Q_{ji}A_j$, hence all the vectors $\hat{x}$ are unit vectors and the cosine of the angle between two vectors $\hat{x}$ and $\hat{y}$ is $u_j(\theta_i):=\frac{Q_{ji}}{Q_{0i}}$, where $d(x,y)=j$.
The map $x\mapsto \hat{x}$ is called a {\em normalized representation} and the sequence $(u_j(\theta_i))_{j=0}^D$ is called the {\em standard sequence} of $\Gamma$, associated with $\theta_i$. 
As $A U_i=\theta_i U_i$, we have $\theta_i \hat{x}=\sum_{y\sim x} \hat{y}$, and hence the following holds:
\begin{equation}\label{eq:std}
\left\{	
\begin{aligned}
& c_j u_{j-1}(\theta_i)+ a_j u_j(\theta_i)+ b_j u_{j+1}(\theta_i)=\theta_i u_j(\theta_i)\quad(j=1,2,\ldots, D-1), \\
& c_Du_{D-1}(\theta_i)+a_D u_{D}(\theta_i)=\theta_i u_D(\theta_i),
\end{aligned}
\right.
\end{equation}
with $u_0(\theta_i)=1$ and $u_1(\theta_i)=\frac{\theta_i}{k}$. 

\begin{lemma}	\label{lm:biggs}
	{\em (c.f. \cite[Theorem 2.8]{DKT})}
	Let $\Gamma$ be a distance-regular graph with diameter $D$ and $v$ vertices. Let $\theta$ be an eigenvalue of $\Gamma$ and $(u_i)_{i=0}^D$ be the standard sequence associated with $\theta$. Then the multiplicity $m(\theta)$ of $\theta$ as an eigenvalue of $\Gamma$ satisfies
	\begin{align}
		m(\theta) & =\frac{v}{\sum_{i=0}^D k_i u_i^2},                                                         \\
		          & \leq \max\{ \frac{1}{u_1^2},\ldots,\frac{1}{u_{j-1}^2},\frac{\sum_{i=0}^jk_i}{k_ju_j^2}\} \quad(j=1,2,\ldots,D).	\label{eq:biggs2}
	\end{align}
\end{lemma}
\begin{proof}
	We only give a proof of Equation (\ref{eq:biggs2}). 
	\begin{align*}
		\frac{v}{\sum_{i=0}^D k_i u_i^2}                       &=\frac{\sum_{i=0}^D k_i}{\sum_{i=0}^Dk_i u_i^2}\leq \max\{ \frac{\sum_{i=0}^{j-1} k_i}{\sum_{i=0}^{j-1}k_iu_i^2}, \frac{\sum_{i=j}^D k_i}{\sum_{i=j}^Dk_iu_i^2} \}, \\
		\frac{\sum_{i=0}^{j-1} k_i}{\sum_{i=0}^{j-1}k_iu_i^2}  &\leq \max \{\frac{1}{u_1^2},\ldots,\frac{1}{u_{j-1}^2}\},                                                                                                           \\
		\frac{\sum_{i=j}^D k_i}{\sum_{i=j}^Dk_iu_i^2}          &\leq \frac{\sum_{i=j}^D k_i}{k_iu_i^2}.
	\end{align*}
\end{proof}

\begin{lemma}	\label{lm:f}
	{\em \cite[Proposition 4.1.6]{BCN}}
	Let $\Gamma$ be a distance-regular graph with valency $k$ and diameter $D$. Then the following conditions hold
	\begin{enumerate}[i)]
		\item $1=c_1\leq c_2\leq \cdots \leq c_D$,
		\item $k=b_0\geq b_1\geq \cdots \geq b_{D-1}$, 
		\item $k_i$'s ($i=1,2,\ldots,D$) are positive integers,
		\item the multiplicities are positive integers.
	\end{enumerate}
\end{lemma}

\begin{lemma}	\label{lm:pu}
	{\em (c.f. \cite[Proposition 3.1]{BKP15})}
	Let $\g$ be a non-bipartite distance-regular graph with valency $k$ and odd girth $g=2t+1$. Then
	\begin{align}		\label{eq:1}
		\sum_{i=0}^t p_i(\eta) u_i\geq 0, 
	\end{align}	
	where $(u_i)_{i=0}^D$ is the standard sequence associated with the smallest eigenvalue $\theta_{\min}$, $\eta$ is any eigenvalue of the $g$-gon, and $p_i(x)$ is defined as the following 
	\begin{equation}	\label{eq:rec}
	\begin{split}
	p_0(x) & =1,                                            \\
	p_1(x) & =x,                                            \\
	p_2(x) & =x^2-2,                                        \\
	p_i(x) & = xp_{i-1}(x)-p_{i-2}(x)\quad (i=3,4,\ldots,t).
	\end{split}
	\end{equation}
\end{lemma}
\begin{proof}
	Let $\Delta$ be any $g$-gon in $\g$.  
	Let $B_i$ $(i=0,1,\ldots,t)$ be the matrix with rows and columns indexed by $V(\Delta)$, where the $(v,w)$-entry is $1$ whenever $d_\g (v,w)=i$ and $0$ otherwise. 
	Note that $d_\Delta(v,w)=d_\Gamma(v,w)$ for any two vertices $v,w\in V(\Delta)$, and we have $B_i=p_i(B_1)$ with $p_i(x)$ as Equation (\ref{eq:rec}). 
	By \cite[Proposition 3.1]{BKP15}, for any eigenvalue $\eta$ of $\Delta$, we have Equation (\ref{eq:1}). 
\end{proof}

\begin{lemma}	\label{lm:c2}
	{\em (c.f. \cite[Lemma 5.2]{KQ17})}
	Let $\g$ be a distance-regular graph with valency $k$ and smallest eigenvalue $\theta_{\min}$. 
	If $a_1=0$ and $\theta_{\min}<\frac{12-5k}{7}$, then $c_2\leq 2$. 
\end{lemma}
\begin{proof}
	Choose two vertices $x,y\in V(\g)$ with $d(x,y)=2$. 
	As $a_1=0$, the subgraph induced on ${x,y}\cup\g(x)\cup\g(y)$ is a $K_{2,c_2}$. 
	Let $x\mapsto \hat{x}$ be a normalized representation associated with $\theta=\theta_{\min}$. 
	Consider the Gram matrix of the image of the $K_{2,c_2}$ with the bipartition, we see that 
	$$
	Q=
	\begin{pmatrix}
	\frac{1}{2}(1+u_2) & u_1                         \\
	u_1                & \frac{1}{c_2}(1+(c_2-1)u_2)
	\end{pmatrix}
	$$
	is positive semidefinite, by \cite[Proposition 3.7.1 (iii)]{BCN}. 
	Then $(1,1)Q(1,1)^t\geq 0$, which in turn implies $(u_1 +u_2)((2+c_2) \frac{1-u_2}{u_1+u_2} +4c_2)\geq 0$.
	As $a_1=0$, we see $u_1+u_2=\frac{(\theta+k)(\theta-1)}{k(k-1)}<0$, that is $\frac{4c_2}{2+c_2}\leq- \frac{1-u_2}{u_1+u_2}=\frac{k-\theta}{1-\theta}$. 
	When $k>1$, we have $\theta <\frac{12-5k}{7}<\frac{4-k}{3}$ and $c_2\leq \frac{2k-2\theta}{4-3\theta-k}<3$. 
\end{proof}

\begin{lemma}
	Let $\g$ be a distance-regular graph with valency $k$ and diameter $D$, having smallest eigenvalue $\theta_{\min}$ with associated standard sequence $(u_i)_{i=0}^D$. 
	Then 
	\begin{align}\label{eq:absu}
	|u_{i+1}|\geq \frac{|(\theta_{\min}-a_i)u_i|-c_i|u_{i-1}|}{b_i}\quad (i=0,1,\ldots,D-1).
	\end{align}  
\end{lemma}
\begin{proof}
	By Equation (\ref{eq:std}) we see that $u_{i+1}=\frac{(\theta_{\min}-a_i)u_i-c_i u_{i-1}}{b_i}$ $(i=1,2,\ldots,D)$. 
	As $\theta_{\min}<0$ is the smallest eigenvalue, by \cite[Corollary 4.1.2]{BCN}, we see that $u_{i+1}$, $-u_i$ and $u_{i-1}$ has the same sign. 
	The result follows. 
\end{proof}
	
\section{Main Theorem}	\label{sec:main1}
In this section we will prove our main result. 

\noindent
{\em Proof of Theorem \ref{thm:main1}}

\noindent
	If $g=3$, then $\theta_{\min}\geq -\frac{k}{2}$ by \cite[Proposition 2.11]{DKT}. So we may assume $g\geq 5$. 
	Let $t=\frac{g-1}{2}$ and $\Delta$ be a $g$-gon in $\g$.  
	Let $(u_i)^D_{i=0}$ be the standard sequence associated with the smallest eigenvalue $\theta=\theta_{\min}$. 
	
	Assume $c_t\leq \z k$ for some $\z \leq \frac{1}{2}$. By Lemma \ref{lm:pu}, we have $\sum_{i=0}^t p_i(\eta) u_i\geq 0$, where $p_i(x)$ is as Equation (\ref{eq:rec}).

	We claim that there exist constants $N_i$ such that 
	\begin{equation}	\label{eq:ni}
		|u_i-(\frac{\theta}{k})^i|\leq N_i \z \quad (i=0,1,\ldots,D).
	\end{equation} 
	Note that $u_0=1$ and $u_1=\frac{\theta}{k}$. 
	Assume $|u_i-(\frac{\theta}{k})^i|\leq N_i \z $ for some $1\leq i\leq t-1$. 
	As $c_t\leq \z k$, we see that $c_i\leq c_t\leq \z k$ and $b_i=1-c_i \geq (1-\z )k$. 
	Then
	\begin{align*}
		|u_{i+1}-(\frac{\theta}{k})^{i+1}| & =|\frac{\theta u_{i}-c_i u_{i-1}}{b_i}-(\frac{\theta}{k})^{i+1}|                                                                                                   \\
		                                   & \leq |\frac{\theta}{b_i} u_i-\frac{\theta}{b_i} (\frac{\theta}{k})^i|+|\frac{\theta}{b_i} (\frac{\theta}{k})^i -(\frac{\theta}{k})^{i+1}|+\frac{c_i}{b_i}|u_{i-1}| \\
		                                   & = |\frac{\theta}{b_i}|\cdot |u_i-(\frac{\theta}{k})^i|+ \frac{c_i}{b_i}\cdot|(\frac{\theta}{k})^{i+1}|+\frac{c_i}{b_i}\cdot |u_{i-1}|                              \\
		                                   & \leq |\frac{k}{b_i}| N_i \z + \frac{c_i}{b_i} +\frac{c_i}{b_i}                                                                                                      \\
		                                   & \leq (2N_i +4 )\z ,
	\end{align*}
	where $\frac{k}{b_i}\leq \frac{1}{1-\z }\leq 2$ and $\frac{c_i}{b_i}\leq \frac{\z }{1-\z }\leq 2\z $ $(\z \geq \frac{1}{2})$.
	So we may take $N_0=N_1=0$ and $N_i=2N_{i-1}+4$ $(i=2,3,\ldots,t)$. 
	
	Note that $p_i(\eta)$ is an eigenvalue of the distance-$i$ graph of $\Delta$. 
	Hence $|p_i(\eta)|\leq 2$ $(i=0,1,\ldots,t)$, and by Equation (\ref{eq:ni}), we have 

		\begin{align}
			\sum_{i=0}^t p_i(\eta) u_i & \leq \sum_{i=0}^t p_i(\eta) (\frac{\theta}{k})^i +\sum_{i=0}^t |p_i(\eta)|\cdot |u_i-(\frac{\theta}{k})^i| \notag \\
			                           & \leq \sum_{i=0}^t p_i(\eta) (\frac{\theta}{k})^i + M_1 \z , 		\label{eq:2}
		\end{align}
	where $M_1=\sum_{i=0}^t 2N_i$. 

	By Equation (\ref{eq:rec}), we see that $p_i(x)=\lambda_1^i+\lambda_2^i$ $(i=1,2,\ldots,t)$, with $\lambda_i=\frac{1}{2}(x\pm \sqrt{x^2-4})$ $(i=1,2)$. 
	Define $f(x,y)=\sum_{i=0}^t p_i(x)y ^i=\frac{1-(\lambda_1y)^{t+1}}{1-\lambda_1y}+\frac{1-(\lambda_2y)^{t+1}}{1-\lambda_2y}-1$. 
	Note that the eigenvalues of $\Delta$ are $2\cos \frac{2\pi j}{g}$ $(i=0,1,\ldots, g-1)$. 
	Take $\eta=2\cos \frac{2\pi(t-1)}{g}$, then we see 
	\begin{align}	\label{eq:3}
		f(\eta,-1)=-M_2, 
	\end{align}
	where $M_2=1/ \cos \frac{(t-1)\pi}{g}$. 
	In fact, 
	\begin{align*}
		f(2\cos \frac{2\pi j}{g},-1) & = \frac{1-\big(-e^{2\pi \mathrm i\cdot \frac{ j}{g}} \big)^{t+1}}{1-\big(-e^{2\pi \mathrm i\cdot \frac{ j}{g}}\big)} + \frac{1-\big(-e^{-2\pi \mathrm i\cdot \frac{ j}{g}}\big)^{t+1}}{1-\big(-e^{-2\pi \mathrm i\cdot \frac{ j}{g}}\big)}-1 &  \\
		                             & = (-1)^t \cdot \frac{e^{2\pi \mathrm i\cdot \frac{j(t+1)}{g}}+e^{-2\pi \mathrm i\cdot \frac{jt}{g}}}{1+e^{2\pi \mathrm i\cdot \frac{ j}{g}}}                                                                                                 &  \\
		                             & = (-1)^{t+j}/\cos \frac{j\pi}{g}.                                                                                                                                                                                                            &
	\end{align*}

	Take $\z = \min\{\frac{M_2}{2M_1},\frac{1}{2}\}$. 
	Note that $M_1$, $M_2$, and hence $\z $ is determined by $g$. 
	By Equation (\ref{eq:3}), we see $f(\eta,-1)+M_1\z \leq -\frac{M_2}{2}<0$. 
	We also have $f(\eta,0)+M_1\z = 1+M_1\z >0$. 
	By Equation (\ref{eq:1}) and (\ref{eq:2}), we have $0\leq f(\eta,\frac{\theta}{k})+M_1\z $. 
	Take $-(1-\varepsilon_1(\z ))$ as the smallest root $y$ of the equation  $f(\eta,y)+M_1\z = 0$ in the interval $(-1,0)$.  
	It follows that $\theta\geq -(1-\varepsilon_1(\z ))k$. 
	
	Now we consider the case $c_t> \z k$. 
	
	If $c_t>1$, then we claim that the diameter $D\leq \frac{4t}{\z ^2}$ and $\theta_{\min}\geq -(1-\varepsilon_2(\z ))k$ for some constant $\varepsilon_2(\z )>0$. 	
	Without loss of generality, we may assume $4^i t\leq D\leq 4^{i+1}t$ for some integer $i\geq 1$. 
	If $c_{2t-1+j}=c_{2t-1}\neq 1$, by \cite[Theorem 7.1]{DKT}, we see $j\leq 2t-1$, that is $c_{4t-1}>c_{2t-1}$. 
	Then $c_{4t-1}=c_{2t+j}>c_{2t-1+j}$ for some $0\leq j\leq 2t-1$, and $c_{4t-1}\geq 2c_t$ by \cite[Proposition 7.2]{DKT}. 
	This implies $k\geq c_{4^i t}\geq 2^i c_t$, that is $D \leq 4t (\frac{k}{c_t})^2\leq \frac{4t}{\z ^2}$. 	
	Then by \cite[Theorem 1.1]{KQ17}, the set $S$ of distance-regular graphs with valency $k$, diameter $D\leq \frac{4t}{\z ^2}$, smallest eigenvalue $\theta_{\min}\leq -(1-\varepsilon_1(\z ))k$ and odd girth $g$ is finite. 
	Take
	\begin{align*}
	\varepsilon_2(\z )=
	\left\{
		\begin{aligned}
		 & \min_{\g\in S}\frac{k+\theta_{\min}}{k},~ & &\text{if $S\neq \emptyset$}, \\
		 & \varepsilon_1(\z ),~                            & &\text{otherwise}.
		\end{aligned}
	\right.
	\end{align*}
	
	If $c_t=1$, then $k< \frac{1}{\z }$. The set $S'$ of distance-regular graphs with valency $k <\frac{1}{\z }$ and odd girth $g$ is finite, by \cite[Theorem 1.1]{BDK15}. 
	Take 
	\begin{align*}
	\varepsilon_3(\z )=
	\left\{
		\begin{aligned}
		 & \min_{\g\in S'}\frac{k+\theta_{\min}}{k},~ & &\text{if $S'\neq \emptyset$}, \\
		 & \varepsilon_1(\z ),~                             & &\text{otherwise}.
		\end{aligned}
	\right.
	\end{align*}
	
	Take $\varepsilon = \min\{\varepsilon_1(\z ),\varepsilon_2(\z ), \varepsilon_3(\z )\}$ and the result follows. 		
\ \ \ \ \ \ \qed
\begin{remark}
	When the odd girth $g=5$ and $c_2\leq \z k$, we may take $N_2=\frac{2}{1-\z}$. Then $f(x, y)+M_1\z=1+xy+(x^2-2)y^2+\frac{4\z}{1-\z}$. 
	By substituting $\eta=2 \cos \frac{2\pi}{5}$ into $f(\eta,\frac{\theta}{k})+M_1\z\geq 0$, we find an inequality between $\z$ and $\frac{\theta}{k}$. 
	For example, if $\z=0.1$, then $\theta\geq -0.78k$. 
\end{remark}

\section{Distance-regular graphs with relatively small $\theta_{\min}$}
In this section we study distance-regular graphs with relatively small $\theta_{\min}$.
In the rest of this section we will give a proof of Theorem \ref{thm:main2}.

\noindent
{\em Proof of Theorem \ref{thm:main2}}

\noindent
	Assume $\g$ has odd girth $g=2t+1$. 
	Let $(u_i)_{i=0}^D$ be the standard sequence associated with the smallest eigenvalue $\theta=\theta_{\min}$. 
	
	We first consider the case $D=4$. 
 	We may assume $k\geq 5$, otherwise $\g$ is the 9-gon or the Coxeter graph by \cite{BBS86} and \cite[Theorem 1.1]{BK99}. 
		
		As $\theta<-\frac{k}{2}$, by \cite[Proposition 2.11]{DKT}, we have $a_1=0$. 
		If $a_2\neq 0$, that is $t=2$, then substitute $\eta=2\cos\frac{2\pi (t-1)}{g}$ into Equation (\ref{eq:1}) and we get 
		$\frac{(k-t) (2 k+\sqrt{5} t+t+\sqrt{5}-1)}{2k(k-1) }\geq 0$, which implies that $\theta \geq \frac{-2 k-\sqrt{5}+1}{\sqrt{5}+1}$. 
		Combine it with $\theta\leq -\frac{3}{4}k$, we see that $k\leq 2$. Hence $a_2=0$. 
		Note that $\theta\leq -\frac{3}{4}k< \frac{12-5k}{7}$, by Lemma \ref{lm:c2}, we see $c_2\leq 2$.
				
		If $a_3\neq 0$, then consider 
		\begin{equation}	\label{eq:a3}
			\left\{
				\begin{aligned}
					\sum_{i=0}^t p_i(\eta)u_i & \geq 0      \\
					-\frac{D-1}{D}k        & \geq \theta
				\end{aligned}
			\right.
		\end{equation}
		with $\eta=2$, we obtain that $k\leq 4$ if $c_2=1$, and $k\leq 8$ if $c_2=2$. 
		No intersection arrays satisfy Lemma \ref{lm:f}, with $5\leq k\leq 8$, $D=4$, $a_1=a_2=0\neq a_3$, $c_2=2$ and $\theta_{\min}\leq -\frac{3}{4}k$. 
		Hence $a_3=0$. 
		
		Assume $k\geq 36$. 
		Since $k\geq 36$, $c_2\leq 2$ and $\theta\leq -\frac{3}{4}k$, by Equation (\ref{eq:absu}), we obtain $|u_2|\geq 0.5500$ and $|u_3|\geq 0.3926$. 
		Now we consider the intersection matrix $L$ of $\g$, where
		$$
		L=\begin{pmatrix}
			0 & k   & 0   & 0     & 0     \\
			1 & 0   & k-1 & 0     & 0     \\
			0 & c_2 & 0   & k-c_2 & 0     \\
			0 & 0   & c_3 & 0     & k-c_3 \\
			0 & 0   & 0   & c_4   & k-c_4
		\end{pmatrix}. 
		$$
		We see that $k^2+\theta^2 \leq tr(L^2) \leq k^2 + 6k+ c_4(2k-c_4)$, where $c_2\leq 2$ and $c_3\leq c_4$. 
		Since $k\geq 36$ and $\theta\leq -\frac{3}{4}k$, we obtain that $\frac{c_4}{k}\geq 0.2227$. 
		By Lemma \ref{lm:biggs}, we see that $m\leq \max\{\frac{1}{u_1^2},\frac{1}{u_2^2},\frac{k_3+k_4}{k_3u_3^2}\}$. 
		Since $k_3b_3=k_4c_4$, we see $\frac{k_3+k_4}{k_3u_3^2}\leq \frac{1}{u_3^2}(1+\frac{k}{c_4})$. 
		With $|u_1|\geq \frac{3}{4}$, $|u_2|\geq 0.5500$, $|u_3|\geq 0.3926$ and $\frac{c_4}{k}\geq 0.2227$, we obtain $m<36$.  
		By \cite[Theorem 4.4.4]{BCN}, we see that $k\leq m<36$, a contradiction. 
		It follows that $k\leq 35$. 
		Then we check the intersection arrays satisfy Lemma \ref{lm:f}, with $5\leq k\leq 35$, $D=4$, $a_1=a_2=a_3=0\neq a_4$, $c_2=1$ or $2$ and $\theta_{\min}\leq -\frac{3}{4}k$, and we get the folded $9$-cube and odd graph $O_5$. 
		This shows the case $D=4$. 
	
	Now we consider the case $D=5$. 
	Similar to the case $D=4$, we may assume $k\geq 5$, otherwise $\g$ is the $11$-gon, by \cite{BBS86} and \cite[Theorem 1.1]{BK99}. 
	As $\theta<-\frac{k}{2}$, by \cite[Proposition 2.11]{DKT}, we have $a_1=0$. 
	Substitute $\eta=2\cos\frac{2\pi (t-1)}{g}$ with $t=2$ into Equation (\ref{eq:1}), we obtain $\theta \geq \frac{-2 k-\sqrt{5}+1}{\sqrt{5}+1}$. 
	Together with $\theta\leq -\frac{4}{5}k$, we see $k\leq 2$, and hence $a_2=0$.  
	Since $\theta\leq -\frac{4}{5}k<\frac{12-5k}{7}$, by Lemma \ref{lm:c2}, we have and $c_2\leq 2$. 
	
	If $a_3\neq 0$, then consider Equation (\ref{eq:a3}) with $\eta=2$, we obtain that $k\leq 3$ if $c_2=1$, and $k\leq 5$ if $c_2=2$. 
	By \cite[Theorem 1.13.2]{BCN}, no such graphs exist with $k=5$ and $c_2=2$. Hence $a_3=0$. 
	
	We consider $a_4\neq 0$. 
	If $c_3\leq 0.3750k$, combine it with Equation (\ref{eq:a3}),  where  $\eta=-1$ ($g=9$), we see that $k\leq 24$. 
	Assume $k\geq 24$, then $c_3\geq 0.3750k$. 
	By Equation (\ref{eq:absu}), we obtain $|u_2|\geq 0.6243$ and $|u_3|\geq 0.4721$. 
	Note $\frac{k_4}{k_3}=\frac{b_3}{c_4}\leq \frac{1-c_3}{c_3}$ and $\frac{k_5}{k_3}=\frac{b_3b_4}{c_4c_5}\leq (\frac{1-c_3}{c_3})^2$. 
	By Lemma \ref{lm:biggs}, we see that 
	\begin{equation}	\label{eq:m}
	\begin{aligned}
		m & \leq \max\{\frac{1}{u_1^2},\frac{1}{u_2^2},\frac{1}{u_3^2}(1+\frac{1-c_3}{c_3}+(\frac{1-c_3}{c_3})^2)\},
	\end{aligned}
	\end{equation}
	that is $k\leq m\leq 24$ (Lemma \cite[Theorem 4.4.4]{BCN}). 
	No intersection arrays satisfy Lemma \ref{lm:f} with $5\leq k\leq 24$, $D=5$, $a_1=a_2=a_3=0\neq a_4$, $c_2=1$ or $2$ and $\theta\leq -\frac{4}{5}k$. 
	Hence $a_4=0$. 
	
	Assume $k\geq 71$. Then by Equation (\ref{eq:absu}), we see that $|u_2|\geq 0.6348$ and $|u_3|\geq 0.4994$,  where $\theta_1\leq -\frac{4}{5}k$, $c_2=1$ or $2$. 
	Then as $m\geq k\geq 71$, by Equation (\ref{eq:m}), we obtain $c_3\leq 0.2166 k$. It implies $|u_4|\geq 0.3344$ by Equation (\ref{eq:absu}). 
	Consider the intersection matrix $L$, and we see that $k^2+ \theta^2  \leq tr(L^2) \leq k^2+6k+4c_5 k-c_5^2$,	which implies $c_5\geq 0.1440 k$. 
	And we see $k\leq m\leq \min\{\frac{1}{u_1^2},\frac{1}{u_2^2},\frac{1}{u_3^2},\frac{1}{u_4^2}(1+\frac{k}{c_5})\}\leq 71$. 
	It follows that $k\leq 71$. 
	Then we check all intersection arrays satisfy Lemma \ref{lm:f} with $5\leq k\leq 71$, $D=5$, $a_1=a_2=a_3=a_4=0\neq a_5$, $c_2=1$ or $2$ and $\theta\leq -\frac{4}{5}k$ and we obtain the odd graph $O_6$ and the folded $11$-cube. This shows the case $D=5$. 
\ \ \ \ \qed

\vspace{5mm}
\noindent
{\bf Acknowledgments}\\
JHK is partially supported by the National Natural Science Foundation of China (Grant No. 11471009 and Grant No. 11671376) and by 'Anhui Initiative in Quantum Information Technologies' (Grant No. AHY150200). 
ZQ is partially supported by the National Natural Science Foundation of China (Grant No. 11801388).

\bibliographystyle{plain}
\bibliography{ref}
\end{document}